\documentclass{article}
\usepackage{amsmath, amsthm}
\usepackage{amssymb}
\newtheoremstyle{theorem}
  {10pt}          
  {10pt}  
  {\sl}  
  {\parindent}     
  {\bf}  
  {. }    
  { }    
  {}     
\newtheorem{thm}{Theorem}[section]

\newtheorem{prop}[thm]{Proposition}
\newtheorem{defn}{Definition}[section]
\newtheorem{exam}[defn]{Example}
\newtheorem{rem}{Remark}[section]
\begin{document}

\thispagestyle{empty} \setcounter{page}{1}



\begin{center}
{\Large\bf Integration by parts and its applications of a new nonlocal fractional derivative with Mittag-Leffler nonsingular kernel}

\vskip.20in
 T. Abdeljawad$^{a}$, D. Baleanu$^{b,c}$ \\[2mm]
{\footnotesize $^{a}$Department of Mathematics and Physical Sciences,
Prince Sultan
University\\ P. O. Box 66833,  11586 Riyadh, Saudi Arabia\\
 $^{b}$Department of Mathematics and Computer Science, \c{C}ankaya University\\ 06530 Ankara, Turkey\\
 email: dumitru@cankaya.edu.tr\\
 $^{c}$ Institute of Space Sciences, Magurele-Bucharest, Romania}

\end{center}

\vskip.2in

{\footnotesize {\noindent Abstract.

 In this manuscript we define the right fractional derivative and its corresponding right fractional integral for the recently introduced nonlocal fractional derivative with Mittag-Leffler kernel. Then, we obtain the related integration by parts formula. We use  the $Q-$operator to confirm our results. The corresponding  Euler-Lagrange equations are obtained and one illustrative example is discussed.}
 \\

{\bf Keywords}: \emph{fractional calculus,Mittag-Leffler function,fractional integration by parts,\\
fractional Euler-Lagrange equations}

\section{Introduction}

Fractional calculus is developing faster during the last few years and many phenomena possessing the power law effect were described accurately with fractional models \cite{podlubny, Samko, Magin, Kilbas,  6, 7,8,9}.
Many excellent results of the fractional models  were reported in various fields of science and engineering. One of the specificity of the fractional calculus is that we have many fractional derivatives which gives the researcher the opportunity to choose  the specific fractional derivative which corresponds better to a given real world problem.
The description of phenomena with memory effect is still a big challenge for the researchers, therefore new tools  and methods should be created to be able to get better description of the real world phenomena and the existing models. In this respect it seems that there is a need of new fractional derivatives with nonsingular kernel.
One of  the best candidates among the existing kernels is the one based on Mittag-Leffler(ML) functions \cite{Abdon}.
Based on this, very recently a new fractional derivative \cite{Abdon} was constructed and applied to several real world problems \cite{Abdon1,Abdon2}. For the nonlocal fractional derivatives with nonsingular  exponential
 kernel we refer to \cite{FCaputo, Losada} and for other local approaches of the fractional derivatives we refer to the recent manuscripts \cite{Roshdi,Thconformable}.
In this paper we would like to present several important properties of the new derivative introduced in  \cite{Abdon} in order to see the advantages of it as well as in order to start to apply it in fractional variational principles and optimal control problems.
Having above mentioned thinks in mind we  present in the first chapter the fundamental integration by parts formula. Integration by parts is of great importance in  fractional calculus \cite{Kilbas} and discrete fractional calculus \cite{TDbyparts,THFer, Thsh,ThdualCaputo}. In the third chapter we developed the corresponding fractional Euler-Lagrange equations and we give an illustrative example of it.

\indent
From the classical fractional calculus, we recall

\begin{itemize}
  \item The left Riemann-Liouville fractional of order $\alpha >0$ starting from  $a$ is defined by
   $$(~_{a}I^\alpha f)(t)=\frac{1}{\Gamma(\alpha)}\int_a^t (t-s)^{\alpha-1}f(s)ds.$$
  \item The right Riemann-Liouville fractional of order $\alpha >0$ ending at $b>a$ is defined by
   $$(I_b^\alpha f)(t)=\frac{1}{\Gamma(\alpha)}\int_t^b (s-t)^{\alpha-1}f(s)ds.$$
  \item The left Riemann-Liouville fractional derivative of order $0<\alpha <1$ starting at $a$ is defined by
  $$(~_{a}D^\alpha f)(t)=\frac{d}{dt}(~_{a}I^{1-\alpha} f)(t).$$
  \item The right Rieemann-Liouville fractional derivative of order $0<\alpha <1$ ending  at $b$ is defined by
  $$(D_b^\alpha f)(t)=\frac{-d}{dt}(I_b^{1-\alpha} f)(t).$$
\end{itemize}

\section{The right fractional derivative and integration by parts formula}

If $f$ is defined on an interval $[a,b]$, then the action of the $Q-$operator is defined as $(Q f)(t)=f(a+b-t)$.
\begin{defn} \cite{Abdon}
Let $f \in H^1(a,b),~~a<b,~~\alpha \in [0,1]$, then the definition of the new (left Caputo) fractional derivative in the sense of Abdon and Baleanu is defined by:
\begin{equation}\label{d1}
 ( ~~^{ABC}~_{a}D^\alpha f)(t)=\frac{B(\alpha)}{1-\alpha} \int_a^t f^\prime(x)E_\alpha[-\alpha \frac{(t-x)^\alpha}{1-\alpha}]dx
\end{equation}

and in the left Riemann-Liouville sense by:
\begin{equation}\label{d2}
 ( ~~^{ABR}~_{a}D^\alpha f)(t)=\frac{B(\alpha)}{1-\alpha}\frac{d}{dt} \int_a^t f(x)E_\alpha[-\alpha \frac{(t-x)^\alpha}{1-\alpha}]dx.
\end{equation}
The associated fractional integral by
\begin{equation}\label{d3}
  (~^{AB}_{a}I^\alpha f)(t)=\frac{1-\alpha}{B(\alpha)}f(t)+\frac{\alpha}{B(\alpha)}(~_{a}I^\alpha f)(t).
\end{equation}
\end{defn}
Let's denote the new right Riemann-Liouville fractional derivative that we wish to propose  by $~^{ABR}D^\alpha_b$ and its corresponding integral by $~^{AB}I^\alpha_b$.
From classical fractional calculus it is known that $(_{a}I^\alpha Qf)(t)=Q (I_b^\alpha f)(t)$ and $(_{a}D^\alpha Qf)(t)=Q (D_b^\alpha f)(t)$. We wish this relation to be satisfied for the new left and right fractional derivatives and integrals.

\begin{eqnarray} \label{d4}
 \nonumber
  (~_{a}~^{ABR}D^\alpha Qf)(t) &=& \frac{B(\alpha)}{1-\alpha} \frac{d}{dt} \int_a^t f(a+b-x)E_\alpha [-\alpha \frac{(t-x)^\alpha}{1-\alpha}]dx\\  \nonumber
   &=& \frac{B(\alpha)}{1-\alpha} \frac{d}{dt} \int_{a+b-t}^b f(u)E_\alpha [-\alpha \frac{(u-(a+b-t))^\alpha}{1-\alpha}]dx, \\
\end{eqnarray}
where the change of variable $u=a+b-x$ is used. The relation (\ref{d4}) suggests the following definition for the new right fractional derivative:
\begin{defn} The right fractional new derivative with ML kernel of order $\alpha \in [0,1]$ is defined by
$$(~^{ABR}D^\alpha_b f)(t)=-\frac{B(\alpha)}{1-\alpha}\frac{d}{dt} \int_t^b f(x)E_\alpha[-\alpha \frac{(x-t)^\alpha}{1-\alpha}]dx.$$

\end{defn}
On the other hand,

\begin{eqnarray} \label{dd}
 \nonumber
  (~^{AB}~_{a}I^\alpha Qf)(t) &=& \frac{1-\alpha}{B(\alpha}f(a+b-t)+\frac{\alpha}{B(\alpha)} (~_{a}I^\alpha Qf)(t) \\  \nonumber
   &=& \frac{1-\alpha}{B(\alpha}f(a+b-t)+\frac{\alpha}{B(\alpha)} Q(I_b^\alpha f)(t) \\
   &=& Q[\frac{1-\alpha}{B(\alpha}f(t)+\frac{\alpha}{B(\alpha)} (I_b^\alpha f)(t)].
\end{eqnarray}

Moreover,  we solve the equation $(~^{AB}D^\alpha_b f)(t)=u(t)$. Indeed,
\begin{eqnarray}
 \nonumber
  (~^{AB}D^\alpha_b f)(t) &=& (~^{AB}D^\alpha_b Q Qf)(t)=(Q~^{AB}~_{a}D^\alpha Qf)(t)=u(t), \\
\end{eqnarray}
or
$$(~^{AB}~_{a}D^\alpha Qf)(t)=Qu(t),$$

and hence, $$Qf(t)=\frac{1-\alpha}{B(\alpha)}Qu(t)+ \frac{\alpha}{B(\alpha)} ~_{a}I^\alpha Q u(t)=\frac{1-\alpha}{B(\alpha)}Qu(t)+ \frac{\alpha}{B(\alpha)}Q I_b^\alpha  u(t).$$
Applying $Q$ to both sides above, we have
\begin{equation}\label{d5}
 f(t)=\frac{1-\alpha}{B(\alpha)}u(t)+ \frac{\alpha}{B(\alpha)} I_b^\alpha  u(t).
\end{equation}
Now, relations (\ref{dd}) and (\ref{d5}) suggest the following definition for the new right fractional integral:
\begin{defn} \label{right integral}
The right fractional new integral with ML kernel of order $\alpha \in [0,1]$ is defined by
$$(~^{AB}I_b^\alpha f)(t)=\frac{1-\alpha}{B(\alpha)}f(t)+ \frac{\alpha}{B(\alpha)} I_b^\alpha  f(t)$$
\end{defn}

Before we present an integration by part formula for the new proposed fractional derivatives and integrals we introduce the following function spaces:
For $p\geq1$ and $\alpha >0$, we define

\begin{equation}\label{n1}
  (~^{AB}~_{a}I^\alpha (L_p)=\{f: f=~^{AB}~_{a}I^\alpha \varphi, ~~\varphi \in L_p(a,b)\}.
\end{equation}
and
\begin{equation}\label{n2}
  (~^{AB}I_b^\alpha (L_p)=\{f: f=~^{AB}I_b^\alpha \phi, ~~\phi \in L_p(a,b)\}.
\end{equation}

\indent

In \cite{Abdon} it was shown that the left fractional   operator $~^{ABR}~_{a}D^\alpha$ and its associate fractional integral $~^{AB}~_{a}I^\alpha$ satisfy $(~^{ABR}~_{a}D^\alpha {AB}~_{a}I^\alpha f)(t)=f(t)$ and above we have shown that  $(~^{ABR}D_b^\alpha ~^{AB}I_b^\alpha f)(t)=f(t)$ . On the other we next prove that $(~
^{AB}~_{a}I^\alpha ~^{ABR}~_{a}D^\alpha  f)(t)=f(t)$ and $(~^{AB}I_b^\alpha ~^{ABR}D_b^\alpha  f)(t)=f(t)$ and hence the function spaces $(~^{AB}~_{a}I^\alpha (L_p)$ and $(~^{AB}I_b^\alpha (L_p)$ are nonempty.

\begin{thm}
The functions $(~^{ABR}~_{a}D^\alpha  f)(t)$ and $(~^{ABR}D_b^\alpha  f)(t)$ satisfy the equations
$$(~^{AB}~_{a}I^\alpha g)(t)=f(t),~~~~~~(~^{AB}I_b^\alpha g)(t)=f(t),$$ respectively.
\end{thm}
\begin{proof}
We just prove the left case. The right case can be proved by means of the $Q-$operator. From, the definition the first equation is equivalent to

$$\frac{1-\alpha}{B(\alpha)}g(t)+\frac{\alpha}{B(\alpha)} (~_{a}I^\alpha g)(t)=f(t).$$

Apply the Laplace transform to see that

$$\frac{1-\alpha}{B(\alpha)}G(s)+\frac{\alpha}{B(\alpha)} s^{-\alpha} G(s)=F(s).$$
From which it follows that

$$G(s)=\frac{B(\alpha)}{1-\alpha} \frac{F(s) s^{\alpha}}{s^\alpha+ \frac{\alpha}{1-\alpha}}.$$
Finally, the Laplace inverse will lead to that $g(t)=(~^{ABR}~_{a}D^\alpha  f)(t)$.
\end{proof}
\begin{thm} (Integration by parts)\label{Integration by parts}
Let $\alpha >0$, $p\geq 1,~q \geq 1$, and $\frac{1}{p}+\frac{1}{q}\leq 1+\alpha$ ($p\neq1$  and $q\neq1$ in the case $\frac{1}{p}+\frac{1}{q}=1+\alpha$ ). Then
\begin{itemize}
  \item If  $\varphi(x) \in L_p(a,b) $ and $\psi(x) \in L_q(a,b)$ , then
   \begin{eqnarray} \label{IBP 1}
            \nonumber
             \int_a^b \varphi(x) (~^{AB}~_{a}I^\alpha\psi)(x)dx &=& \frac{1-\alpha}{B(\alpha)}\int_a^b \psi(x)\varphi(x)dx+ \frac{\alpha}{B(\alpha)}\int_a^b(I_b^\alpha\varphi)(x) \psi(x)dx \\
              &=& \int_a^b \psi(x) (~^{AB}I_b^\alpha\varphi(x)dx
           \end{eqnarray}

      and similarly,
      \begin{eqnarray}
        \int_a^b \varphi(x) (~^{AB}I_b^\alpha\psi)(x)dx&=& \frac{1-\alpha}{B(\alpha)}\int_a^b \psi(x)\varphi(x)dx+ \frac{\alpha}{B(\alpha)}\int_a^b(~_{a}I^\alpha\varphi)(x) \psi(x)dx \\
         &=&  \int_a^b \psi(x) (~^{AB}~_{a}I^\alpha\varphi)(x)dx
      \end{eqnarray}

  \item If $f(x) \in ~^{AB}I_b^\alpha (L_p) $ and $g(x) \in ~^{AB}~_{a}I^\alpha (L_q)$, then  $$\int_a^b f(x) (~^{ABR}~_{a}D^\alpha g)(x)dx=\int_a^b (~^{ABR}D_b^\alpha f)(x) g(x)dx$$
\end{itemize}
\end{thm}
\begin{proof}
\begin{itemize}
  \item From the definition and the integration by parts for( classical) Riemann-Liouville fractional integrals we have
  \begin{eqnarray}
         \nonumber
          \int_a^b \varphi(x) (~^{AB}~_{a}I^\alpha\psi)(x)dx &=& \int_a^b \varphi(x) [\frac{1-\alpha}{B(\alpha)}\psi(x)+\frac{\alpha}{B(\alpha)}~_{a}I^\alpha \psi(x)]dx  \\ \nonumber
           &=&\frac{1-\alpha}{B(\alpha)} \int_a^b\varphi(x) \psi(x)  dx+\frac{\alpha}{B(\alpha)}\int_a^b \psi (x)I_b^\alpha \varphi(x) dx \\ \nonumber
           &=& \int_a^b \psi(x)[\frac{1-\alpha}{B(\alpha)} \varphi(x) +\frac{\alpha}{B(\alpha)} I_b^\alpha \varphi (x)]dx \\
           &=&\int_a^b \psi(x) (~^{AB}I_b^\alpha\varphi(x)dx.
  \end{eqnarray}
  The other case follows similarly by Definition \ref{right integral} and the integration by parts for( classical) Riemann-Liouville fractional integrals.
  \item From definition and the first part we have
  \begin{eqnarray}
   \nonumber
    \int_a^b f(x) (~^{ABR}~_{a}D^\alpha g)(x)dx &=& \int_a^b ( ~^{AB}I_b^\alpha \phi)(x).(~ ~^{ABR}~_{a}D^\alpha  \circ~^{ABR} ~_{a}I^\alpha \varphi)(x)dx\\ \nonumber
     &=& \int_a^b ( ~^{AB}I_b^\alpha \phi)(x). \varphi(x)dx \\ \nonumber
     &=&  \frac{1-\alpha}{B(\alpha)} \int_a^b \phi(x)\varphi(x)dx +\frac{\alpha}{B(\alpha)} \int_a^b  \phi(x) (~_{a}I^\alpha \varphi)(x)dx\\ \nonumber
     &=& \frac{1-\alpha}{B(\alpha)} \int_a^b (~^{ABR}D^\alpha _b f)(x)(~^{ABR}~_{a}D^\alpha g)dx+ \\ \nonumber
     &+&\frac{\alpha}{B(\alpha)} \int_a^b (~^{ABR}D^\alpha _b f)(x)[\frac{B(\alpha)}{\alpha}g(x)- \frac{1-\alpha}{\alpha}(~^{ABR}~_{a}D^\alpha g)]dx \\ \nonumber
      &=& \int_a^b (~^{ABR}D_b^\alpha f)(x) g(x)dx.
     \end{eqnarray}
     In the proof, the identity $(~_{a}I^\alpha \varphi )(x)=\frac{B(\alpha)}{\alpha}(~^{AB}~_{a}I^\alpha \varphi )(x)-\frac{1-\alpha}{\alpha} \varphi(x) $ derived from (\ref{d3}) is used.
\end{itemize}

\end{proof}

\begin{exam}This example is a numerical application of Theorem \ref{Integration by parts}.
\begin{itemize}
  \item To verify (\ref{IBP 1}), let $\psi(x)=x$, $\varphi(x)=1-x$, $\alpha=\frac{1}{2}$, $[a,b]=[0,1]$ and $B(\alpha)=1$. Then,
  $$~^{AB}~_{0}I^{1/2}x=\frac{x}{2}+\frac{1}{2} \frac{\Gamma(2)x^{3/2}}{\Gamma(5/2)}=\frac{x}{2}+\frac{2x^{3/2}}{3\sqrt{\pi}},$$
  and
  $$~^{AB}I_1^{1/2}(1-x)=\frac{1-x}{2}+\frac{2 (1-x)^{3/2}}{3\sqrt{\pi}}.$$
  Hence, the left hand side of (\ref{IBP 1}) results in
  \begin{equation}\label{ver1}
    \int_a^b \varphi(x) (~^{AB}~_{a}I^\alpha \psi)(x)dx=\int_0^1 (1-x)~^{AB}{~_0}I^{1/2}x=\int_0^1 (1-x)[\frac{x}{2}+\frac{2x^{3/2}}{3\sqrt{\pi}}] dx=\frac{1}{12}+\frac{8}{105\sqrt{\pi}},
  \end{equation}
  and
 \begin{equation}\label{ver2}
    \int_a^b \psi(x) (~^{AB}I_b^\alpha\varphi(x)dx=\int_0^1 x (~^{AB}I_1^{1/2}(1-x)dx=\int_0^1x[\frac{1-x}{2}+\frac{2 (1-x)^{3/2}}{3\sqrt{\pi}}]dx=\frac{1}{12}+\frac{8}{105\sqrt{\pi}}.
  \end{equation}

  \item To verify the second part of Theorem \ref{Integration by parts}, let $f(x)=\frac{1-x}{2}+\frac{2 (1-x)^{3/2}}{3\sqrt{\pi}}$ and $g(x)=\frac{x}{2}+\frac{2x^{3/2}}{3\sqrt{\pi}}$, with $\alpha=\frac{1}{2}$, $[a,b]=[0,1]$ and $B(\alpha)=1$.
      Then,
      $$\int_a^b f(x) (~^{ABR}~_{a}D^\alpha g)(x)dx=\int_0^1[\frac{1-x}{2}+\frac{2 (1-x)^{3/2}}{3\sqrt{\pi}}]xdx=\frac{1}{12}+\frac{8}{105\sqrt{\pi}},$$
      and
    $$\int_a^b (~^{ABR}D_b^\alpha f)(x) g(x)dx=\int_0^1 (1-x)[\frac{x}{2}+\frac{2x^{3/2}}{3\sqrt{\pi}}] dx=\frac{1}{12}+\frac{8}{105\sqrt{\pi}}.$$
\end{itemize}
\end{exam}
From \cite{Abdon} we recall the relation between the Riemann-Liouville and Caputo new derivatives as
\begin{equation}\label{relation C and R}
  (~^{ABC}~_{0}D^\alpha f)(t)=(~^{ABR}~_{0}D^\alpha f)(t)-\frac{B(\alpha)}{1-\alpha} f(0)E_\alpha (-\frac{\alpha}{1-\alpha}t^\alpha)
\end{equation}

From \cite{Antony etal} recall the (left) generalized fractional integral operator
\begin{equation}\label{GIO}
( \textbf{ E}^\gamma_{\rho, \mu, \omega,a^+}\varphi )(x)=\int_a^x (x-t)^{\mu-1} E_{\rho,\mu}^\gamma [\omega (x-t)^\rho]\varphi(t) dt,~~x>a.
\end{equation}
Analogously, the (right) generalized fractional integral operator can be defined by

\begin{equation}\label{GIOr}
( \textbf{ E}^\gamma_{\rho, \mu, \omega,b^-}\varphi )(x)=\int_x^b (t-x)^{\mu-1} E_{\rho,\mu}^\gamma [\omega (t-x)^\rho]\varphi(t) dt,~~x<b,
\end{equation}
where $E_{\rho,\mu}^\gamma (z)=\sum_{k=0}^\infty \frac{(\gamma)_k z^k}{\Gamma(\rho k+\mu)k!},$ is the generalized Mittag-Leffler function which is defined for complex $\rho,\mu,\gamma ~(Re(\rho)>0)$ \cite{Antony etal, Kilbas}.

\begin{defn}
The new (right) Caputo fractional derivative of order $0<\alpha<1$ is defined by
 $$(~^{ABC}D^\alpha_b f)(t)=-\frac{B(\alpha)}{1-\alpha} \int_t^b f^\prime(x)E_\alpha[-\alpha \frac{(x-t)^\alpha}{1-\alpha}]dx,$$
\end{defn}
Next, we prove the right version of (\ref{relation C and R}) by making use of the $Q-$operator.
\begin{prop}The right new Riemann-Liouville fractional derivative and the new right Caputo fractional derivative are related by the identity:
\begin{equation}\label{right relation C and R}
  (~^{ABC}~D_b^\alpha f)(t)=(~^{ABR}D_b^\alpha f)(t)-\frac{B(\alpha)}{1-\alpha} f(b)E_\alpha (-\frac{\alpha}{1-\alpha}(b-t)^\alpha)
\end{equation}
\end{prop}
\begin{proof}
Apply the $Q-$operator to the identity (\ref{relation C and R}) and make use of the dual facts
 $Q(~^{ABR}~_{0}D^\alpha f)(t)=(~^{ABR}D_b^\alpha Qf)(t)$ and $Q(~^{ABC}~_{0}D^\alpha f)(t)=(~^{ABC}D_b^\alpha Qf)(t)$, to obtain that

 $$(~^{ABC}~D_b^\alpha Qf)(t)=(~^{ABR}D_b^\alpha Qf)(t)-\frac{B(\alpha)}{1-\alpha} f(0)E_\alpha (-\frac{\alpha}{1-\alpha}(b-t)^\alpha).$$
 Now replace $f(t)$ by $(Qf)(t)=f(b-t)$ to conclude our claim.

\end{proof}

\begin{prop} (Integration by parts for the Caputo fractional derivative "$(~^{ABC}~_{a}D^\alpha),~~~a=0$") \label{C by parts}
\begin{itemize}
  \item $\int_0^b (~^{ABC}~_{a}D^\alpha f)(t)g(t)= \int_0^b f(t) (~^{ABR}D_b^\alpha g)(t)+ \frac{B(\alpha)}{1-\alpha} f(t) \textbf{ E}^1_{\alpha,1, \frac{-\alpha}{1-\alpha},b^-}g )(t)|_0^b$.
  \item $\int_0^b (~^{ABC}D_b^\alpha f)(t)g(t)= \int_0^b f(t) (~^{ABR}~_{0}D^\alpha g)(t)- \frac{B(\alpha)}{1-\alpha} f(t) \textbf{ E}^1_{\alpha,1, \frac{-\alpha}{1-\alpha},0^+}g )(t)|_0^b$.
  \begin{proof}
  The proof of the first part follows by Theorem \ref{Integration by parts} and (\ref{relation C and R}) and the proof of the second part follows by Theorem \ref{Integration by parts} and  (\ref{right relation C and R}).
  \end{proof}
\end{itemize}

\end{prop}

\section{Fractional Euler-Lagrange Equations}
We prove the Euler-Lagrange equations for a Lagrangian containing the left new Caputo derivative.

\begin{thm}\label{A1}
Let $0<\alpha \leq 1$ be non-integer, $b \in \mathbb{R},~~0<b$,  Assume that the functional $J:C^2[0,b]\rightarrow  \mathbb{R}$ of the form
 $$J(f)=\int_{0}^{b} L(t,f(t),~^{ABC}~_{0}D^\alpha f(t) )dt$$
 has a local extremum in $S=\{y \in C^2[0,b]: ~~y(0)=A, y(b)=B\}$ at some $f \in S$, where
 $L:[0,b]\times \mathbb{R}\times \mathbb{R}\rightarrow \mathbb{R}$. Then,
\begin{equation}\label{E1}
[L_1(s) + ~^{ABR}D_b^\alpha L_2(s)]  =0,~\texttt{for all}~ s \in [0,b],
\end{equation}
where $L_1(s)= \frac{\partial L}{\partial f}(s)$ and $L_2(s)=\frac{\partial L}{\partial ~^{ABC}~_{0}D^\alpha f}(s)$.
\end{thm}
\begin{proof}
Without loss of generality, assume that $J$ has local maximum in $S$ at $f$. Hence, there exists an $\epsilon>0$ such that $J(\widehat{f})-J(f)\leq 0$ for all $\widehat{f}\in S$ with $\|\widehat{f}-f\|=\sup_{t \in \mathbb{N}_a \cap ~_{b}\mathbb{N}} |\widehat{f}(t)-f(t)|< \epsilon$. For any $\widehat{f} \in S$ there is an $\eta \in H=\{y \in C^2[0,b], ~~y(0)=y(b)=0\}$ such that $\widehat{f}=f+\epsilon \eta$. Then, the $\epsilon-$Taylor's theorem implies that
$$L(t,f,\widehat{f})=L(t,f+\epsilon \eta,~^{ABC}~_{0}D^\alpha f+\epsilon ~^{ABC}~_{0}D^\alpha \eta)=L(t,f,~^{ABC}~_{0}D^\alpha f)+ \epsilon [\eta L_1+~^{ABC}~_{0}D^\alpha \eta L_2]+O(\epsilon^2).$$ Then,

\begin{eqnarray}\nonumber
  J(\widehat{f})-J(f) &=& \int_0^bL(t,\widehat{f}(t),~^{ABC}~_{0}D^\alpha \widehat{f}(t))-\int_0^b L(t,f(t),~^{ABC}~_{0}D^\alpha f(t)) \\
  &=& \epsilon \int_0^b[\eta(t) L_1(t)+ (~^{ABC}~_{0}D^\alpha \eta)(t) L_2(t)]+ O(\epsilon^2).
   \end{eqnarray}
Let the quantity $\delta J(\eta,y)=\int_0^b[\eta(t) L_1(t)+ (~^{ABC}~_{0}D^\alpha \eta)(t) L_2(t)]dt$ denote the first variation of $J$.

Evidently, if $\eta \in H$ then $-\eta \in H$, and $\delta J(\eta,y)=-\delta J(-\eta,y)$. For $\epsilon$ small, the sign of $J(\widehat{f})-J(f)$ is determined by the sign of first variation, unless $\delta J(\eta,y)=0$ for all $\eta \in H$. To make the parameter $\eta$ free, we use the integration by part formula in Proposition \ref{C by parts}, to reach
$$\delta J(\eta,y)=\int_0^b \eta (s)[L_1(s) + ~^{ABR}D_b^\alpha L_2(s)]+\eta(t)\frac{B(\alpha)}{1-\alpha} (\textbf{ E}^1_{\alpha,1, \frac{-\alpha}{1-\alpha},b^-}L_2 )(t)|_0^b =0,$$ for all $\eta \in H$, and hence the result follows by the fundamental Lemma of calculus of variation.
\end{proof}
The term $(\textbf{ E}^1_{\alpha,1, \frac{-\alpha}{1-\alpha},b^-}L_2 )(t)|_0^b =0$ above is called the natural boundary condition.

\indent
Similarly, if we allow the Lagrangian to depend on the right Caputo fractional derivative, we can state:
\begin{thm}\label{A2}
Let $0<\alpha \leq 1$ be non-integer, $b \in \mathbb{R},~~0<b$,  Assume that the functional $J:C^2[0,b]\rightarrow  \mathbb{R}$ of the form
 $$J(f)=\int_{0}^{b} L(t,f(t),~^{ABC}D_b^\alpha f(t) )dt$$
 has a local extremum in $S=\{y \in C^2[0,b]: ~~y(0)=A, y(b)=B\}$ at some $f \in S$, where
 $L:[0,b]\times \mathbb{R}\times \mathbb{R}\rightarrow \mathbb{R}$. Then,
\begin{equation}\label{E1}
[L_1(s) + ~^{ABR}~_{0}D^\alpha L_2(s)]  =0,~\texttt{for all}~ s \in [0,b],
\end{equation}
where $L_1(s)= \frac{\partial L}{\partial f}(s)$ and $L_2(s)=\frac{\partial L}{\partial ~^{ABC}D_b^\alpha f}(s)$.
\end{thm}
\begin{proof}
The proof is similar to  Theorem \ref{A1} by applying the second integration by parts in Proposition \ref{C by parts} to get the natural boundary condition of the form $(\textbf{ E}^1_{\alpha,1, \frac{-\alpha}{1-\alpha},0^+}L_2 )(t)|_0^b =0$.
\end{proof}

\begin{thm}\cite{Antony etal}
Let $\rho, \mu,\gamma, \nu, \sigma, \lambda \in \mathbb{C}$ ($Re(\rho), Re(\mu), Re(\nu)>0$), then
\begin{equation}\label{Kil1}
  \int_0^x (x-t)^{\mu-1} E^\gamma_{\rho,\mu}(\lambda [x-t]^\rho) t^{\nu-1} E^\sigma_{\rho,\nu}(\lambda t^\rho)dt=x^{\mu+\nu-1}E^{\gamma+\sigma}_{\rho,\mu+\nu}(\lambda x^\rho).
\end{equation}
In particular, if $\gamma=1,~~\mu=1$  and $\rho=\alpha$, we have
\begin{equation}\label{Kil2}
  \int_0^x  E_{\alpha}(\lambda [x-t]^\alpha) t^{\nu-1} E^\sigma_{\alpha,\nu}(\lambda t^\alpha)dt=x^{\nu}E^{1+\sigma}_{\alpha,1+\nu}(\lambda x^\alpha).
\end{equation}

\end{thm}

From \cite{Kilbas} we recall also the following differentiation formula that will be helpful

For $\alpha, \mu,\gamma, \lambda \in \mathbb{C}$ ($Re(\alpha>0$) and $n \in \mathbb{N}$ we have

\begin{equation}\label{kildif}
  (\frac{d}{dz})^n [z^{\mu-1} E^\gamma_{\alpha, \mu}(\lambda z^\alpha)]=z^{\mu-n-1}E^\gamma_{\alpha, \mu-n}(\lambda z^\alpha),
\end{equation}
Now, by the help of (\ref{Kil2}) and (\ref{kildif}), we have

\begin{equation}\label{ABR derivative of ML}
  ~^{ABR}~_{0}D^\alpha [x^{\nu-1}E^\sigma_{\alpha, \nu}(\lambda x^\alpha ) ]=\frac{B(\alpha)}{1-\alpha} \frac{d}{dx}[x^{\nu}E^{1+\sigma}_{\alpha, 1+\nu}(\lambda x^\alpha) ]=\frac{B(\alpha)}{1-\alpha}x^{\nu-1}E^{1+\sigma}_{\alpha, \nu}(\lambda x^\alpha)
\end{equation}

Similarly, by the help of (\ref{kildif}) and (\ref{Kil2}), we have

\begin{eqnarray}\label{ABC derivative of ML}
 \nonumber
 ~^{ABC}~_{0}D^\alpha [x^{\nu-1}E^\sigma_{\alpha, \nu}(\lambda x^\alpha) ] &=& \frac{B(\alpha)}{1-\alpha} \int_0^x x^{\nu}E_{\alpha}(\lambda (x-t)^\alpha )\frac{d}{dt}[t^{\nu-1}E^\sigma_{\alpha, \nu}(\lambda x^\alpha) ] dt \\
   &=&\frac{B(\alpha)}{1-\alpha}x^{\nu-1}E^{1+\sigma}_{\alpha, \nu}(\lambda x^\alpha).
\end{eqnarray}

\begin{rem}\label{The zero derivative case}
An interesting observation of (\ref{ABR derivative of ML}) and  (\ref{ABC derivative of ML}) is that the function

\begin{eqnarray}
  g(x) &=& \lim_{\nu\rightarrow 0^+}\frac{1-\alpha}{B(\alpha)}x^{\nu-1}E^{-1}_{\alpha, \nu}(\lambda x^\alpha) \\
   &=&  \frac{\alpha x^{\alpha-1}}{B(\alpha) \Gamma(\alpha)},
\end{eqnarray}
 is a nonzero function whose fractional $ABR$ and $ABC$ derivative is zero. This can be seen since $(-1)_0=1,~~(-1)_1=-1$ and $(-1)_k=0$ for $k=2,3,4,...$ and since $$E^{0}_{\alpha, \nu}(\lambda,x)=\frac{x^{\nu-1}}{\Gamma(\nu)}\rightarrow 0,~~\nu\rightarrow 0^+.$$ Note here that the function $g(x)$ tends to the constant function $1$ when $\alpha$ tends to $1$.

\end{rem}

Using the following  relation (14) in \cite{Abdon}
 \begin{equation}\label{a relation of ABC and ABR}
   (~^{ABC}~_{0}D^\alpha f)(t)=(~^{ABR}~_{0}D^\alpha f)(t)-\frac{B(\alpha)}{1-\alpha}f(0) E_\alpha(\lambda t^\alpha),~~\lambda=\frac{-\alpha}{1-\alpha},
 \end{equation}

 and  the identity (see \cite{Kilbas} page 78 for example)
       \begin{equation}\label{f integral of ML functions}
         (~_{0}I^\alpha t^{\beta-1} E_{\mu,\beta} [\lambda  t^\mu](x)=x^{\alpha+\beta-1}E_{\mu,\alpha+\beta} [\lambda  x^\mu],
       \end{equation}
   where  the $ML-$function  with two parameters $\alpha$ and $\beta$ is given by

\begin{equation}\label{cl ML2}
 E_{\alpha,\beta}(z)= \sum_{k=0}^\infty \frac{z^k}{\Gamma(\alpha k+\beta)},~~~(z, \beta \in \mathbb{C};~Re(\alpha)>0),
\end{equation}
where $E_{\alpha,1}(z)=E_{\alpha}(z)$,
we can state the  following result which is very useful tool to solve fractional dynamical systems within Caputo fractional derivative with $ML$ kernals.
\begin{prop}
For $0< \alpha< 1$, we have

\begin{eqnarray}\label{fintegral of Caputo}
 \nonumber
  ( ~^{AB}~_{a}I^\alpha ~^{ABC}~_{a}D^\alpha f)(x) &=& f(x)-f(a)E_{\alpha}(\lambda (x-a)^\alpha)-\frac{\alpha}{1-\alpha}f(a) x^\alpha E_{\alpha,\alpha+1}(\lambda (x-a)^\alpha) \\
   &=& f(x)-f(a).
\end{eqnarray}
 Similarly,

 \begin{equation}\label{rfintegral of Caputo}
   ( ~^{AB}I_b^\alpha ~^{ABC}D_b^\alpha f)(x)=f(x)-f(b)
 \end{equation}
\end{prop}

\begin{exam}

In order to exemplify our results we study an example of physical interest under Theorem \ref{A1}. Namely, let us consider the following fractional  action,

  $J(y)=\int_0^b[\frac{1}{2}( ~^{ABC}~_{0}D^\alpha y(t))^2-V(y(t))],$ where $0<\alpha <1$ and with $y(0),~~y(b)$ are assigned or with the natural boundary condition $(\textbf{ E}^1_{\alpha,1, \frac{-\alpha}{1-\alpha},b^-}~^{ABC}~_{0}D^\alpha y(t) )(t)|_0^b =0 $.  Then,  the Euler-Lagrange equation by applying Theorem \ref{A1} is
       $$(~^{ABR}D_b^\alpha ~o ~~^{ABC}~_{0}D^\alpha y)(s)-\frac{dV}{dy}(s)=0~\texttt{for all}~ s \in [0,b].$$
       Here,  we remark that it is of interest  to deal with the above Euler- Lagrange equations obtained in the above example, where we have composition of  right and left type  fractional derivatives. For such a composition in the classical fractional case together with the action of the $Q-$operator we refer to \cite{TDFdelay}.

       Finally, we solve the above fractional Euler-Lagrange equations for certain potential functions with $\alpha =\frac{1}{2}$, and $B(\alpha)=1$.

       \begin{itemize}
         \item We consider the free particle case $V\equiv 0$: The Euler-Lagrange equations will be reduced to $(~^{ABR}D_b^\alpha ~^{ABC}~_{0}D^\alpha y)(t)=0$. By applying $~^{AB}I_b^\alpha $ to both sides  we reach at
             $$(~^{ABC}~_{0}D^\alpha y)(t)=0.$$
            Then, by Remark \ref{The zero derivative case} with $B(\alpha)=1$ for simplicity (otherwise $B(\alpha)\rightarrow 1$ as $\alpha \rightarrow1 $) , we conclude that
             \begin{equation}\label{sol rep}
               y(t)=c_1+ \frac{\alpha t^{\alpha-1}}{B(\alpha) \Gamma(\alpha)},
             \end{equation}
             and hence using $y(0)=A$, the solution becomes
             \begin{equation}\label{csol rep}
               y(t)=y(0)+ \frac{\alpha t^{\alpha-1}}{B(\alpha) \Gamma(\alpha)},
             \end{equation}
          We remark here that as $\alpha \rightarrow1 $, we get the classical case.
         \item Let $V(y)=c y^2/2$. Then, the fractional Euler-Lagrange equations are become
         $(~^{ABR}D_b^\alpha ~^{ABC}~_{0}D^\alpha  y)(t)=cy(t).$ Then, applying $~^{AB}I_b^\alpha$ and $~^{AB}~_{0}I^\alpha$ respectively together with use of   (\ref{fintegral of Caputo}), we reach at the integral equation
\begin{equation}\label{sol rep2}
           y(t)=y(0)+c (~^{AB}~_{0}I^\alpha ~^{AB}I_b^\alpha y)(t).
           \end{equation}
           \end{itemize}

           Notice that, when $\alpha$ tends to $1$ we get the classical result.
         \end{exam}

\section{Conclusions}

     The fractional derivatives introduced in \cite{Abdon} are of interest for real world problems since they contain nonsingular Mittag-Leffler  kernels. They, obey the calculations done by the $Q-$operator to introduce the right fractional operators.
  We show that the $Q-$operator is an effective tool that helped in defining the right fractional integrals and derivatives and it helps to confirm some  identities by using its dual action.
  The obtained integration by parts formula, in case of the  Caputo derivative in the sense of  Atangana-Baleanu, contains terms expressed by means of the integral operators studied in \cite{Antony etal} whose kernels are  generalized Mittag-Leffler functions.
   The integration by parts formulas produced the corresponding Euler-Lagrange equations under the existence of natural boundary conditions expressed by means of integral operators.
   The obtained formulas such as the integration by parts for the Caputo derivatives  in the left case with $a=0$ and the variational fractional problem with lower limit $0$, all can be generalized by using the Laplace transform starting at $a$ and then applying the $Q-$operator in its general version $(Qf)(t)=f(a+b-t)$ where $a$ can be different from $0$.
 In order to illustrate our results we provided an illustrative example. The results presented in this manuscript can be used successfully for the fractional variational principles and their applications in Physics and Engineering  as well as for control theory.

\end{document}